\documentclass[11pt,oneside,reqno]{amsart}
\usepackage{amssymb,amsmath,amsthm}
\usepackage{a4wide,comment,enumitem,url}
\usepackage{xcolor}
\usepackage{enumitem}
\usepackage{bm}
\usepackage{float}

\usepackage{graphicx}
\usepackage{hyperref}
\usepackage{amsmath,amsopn,amssymb,amsfonts,stmaryrd}
\usepackage{verbatim}
\usepackage{amsthm}
\usepackage{mathtools}
\usepackage{color}
\usepackage{enumitem}
\usepackage[framemethod=TikZ]{mdframed}
\usepackage{bbm}
\usepackage{mathrsfs}
\usepackage{booktabs}
\usepackage{caption}
\usepackage{cancel}
\usepackage{tensor}
\usepackage{cleveref}

\renewcommand{\Im}{{\rm Im}}

\newcommand{\sgn}{\mbox{sgn}}

\newcommand{\SL}{\mathrm{SL}}

\newcommand{\N}{\mathbb N}
\newcommand{\C}{\mathbb C}
\newcommand{\Q}{\mathbb Q}

\theoremstyle{plain}
\newtheorem{thm}{Theorem}[section]
\newtheorem{cor}[thm]{Corollary}
\newtheorem{lem}[thm]{Lemma}

\theoremstyle{definition}

\newtheorem*{rem}{Remark}

\numberwithin{equation}{section}
\numberwithin{thm}{section}

\setlist[enumerate]{leftmargin=*,label=\rm{(\arabic*)}}

\renewcommand{\sgn}{\textnormal{sgn}}

\newcommand{\re}{{\rm Re}}

\renewcommand{\sgn}{{\rm sgn}}
\newcommand{\R}{\mathbb R}
\newcommand{\Z}{\mathbb Z}

\setlist[itemize]{noitemsep, topsep=0pt}

\allowdisplaybreaks

\makeatletter
\newcommand{\vast}{\bBigg@{2}}
\newcommand{\Vast}{\bBigg@{5}}
\makeatother

\title[Prime-detecting quasimodular forms]{On a conjecture about prime-detecting quasimodular forms}
\author{Ben Kane}
\address{The University of Hong Kong, Department of Mathematics, Pokfulam, Hong Kong}
\email{bkane@hku.hk}
\author{Krishnarjun Krishnamoorthy}
	\address{Beijing Institute of Mathematical Sciences and Applications (BIMSA), No. 544, Hefangkou Village, Huaibei Town, Huairou District, Beijing.}
    	\email[Krishnarjun Krishnamoorthy]{krishnarjunmaths@outlook.com}
\author{Yuk-Kam Lau}
\address{Weihai Institute for Interdisciplinary Research, Shandong University, China and Department of Mathematics, The University of Hong Kong,  Pokfulam, Hong Kong}
\email{yukkamlau@hku.hk}
\begin{document}
\date{\today}
\keywords{Quasimodular forms, sign changes of Fourier coefficients}
\subjclass[2020]{11F11,11F30}
\thanks{The research of the first author was supported by grants from the Research Grants Council
of the Hong Kong SAR, China (project numbers HKU 17314122, HKU 17305923).}
\begin{abstract}
	Motivated by weighted partition of $n$ that vanish if and only if $n$ is a prime, Craig, van Ittersum, and Ono conjectured a classification of quasimodular forms which detect primes in the sense that the $n$-th Fourier coefficient vanishes if and only if $n$ is a prime. In this paper, we prove this conjecture by showing that Fourier coefficients of quasimodular cusp forms exhibit infinitely many sign changes. 
\end{abstract}
\maketitle
\section{Introduction}

In \cite{CvIO}, Craig, van Ittersum, and Ono investigate certain weighted partition generating functions that detect primes in the sense that their $n$-th Fourier coefficient vanishes if and only if $n$ is prime. Letting 
\[
\mathcal{U}_a(q)=\sum_{n\geq 1} M_{a}(n) q^n := \sum_{0<s_1<s_2<\dots<s_a} \frac{q^{s_1+s_2+\dots+s_a}}{\left(1-q^{s_1}\right)^2\left(1-q^{s_2}\right)^2\cdots \left(1-q^{s_a}\right)^2}
\]
denote the MacMahon \cite{MacMahon} $q$-series, the numbers $M_a(n)$ give the number of partitions into precisely $a$ different parts, weighted by the multiplicities of each of the parts. In \cite{CvIO}, they show that an integer $n$ is prime if and only if certain linear relations hold for the $M_a(n)$, such as (see \cite[Theorem 1.1 (1)]{CvIO})
\[
\left(n^2-3n+2\right)M_1(n)=8M_2(n).
\]
They investigate these identities within the larger context of quasimodular forms whose $n$-th Fourier coefficient vanishes if and only if $n$ is prime. In particular, they define a subset $\Omega$ of the graded ring of (integer weight) quasimodular forms such that $f\in\Omega$ if and only if for ($q:=e^{2\pi i\tau}$)
\[
f(\tau)=\sum_{n\geq 0} c_f(n) q^n
\]
we have 
\[
c_f(n)\geqslant 0\text{ and }c_f(n)=0\text{ if and only if $n$ is prime}.
\]
Letting $\mathcal{E}$ denote the space of \begin{it}quasimodular Eisenstein series\end{it} (i.e., the vector space spanned by Eisenstein series and their derivatives), they classify $\mathcal{E}\cap\Omega$ in \cite[Theorem 2.3]{CvIO}. Moreover, they conjecture in \cite[Conjecture 2.2]{CvIO} that every element of $\Omega$
 is indeed a quasimodular Eisenstein series. We prove this conjecture here.
\begin{thm}\label{thm:PrimeDetecting}
We have $\Omega=\mathcal{E}\cap \Omega$.
\end{thm}
Combining Theorem \ref{thm:PrimeDetecting} with \cite[Theorem 2.3]{CvIO}, one obtains a full classification of quasimodular forms which detect primes. Letting $D:=\frac{1}{2\pi i} \frac{d}{d\tau} = q\frac{d}{dq}$, following \cite[(2.3)]{CvIO} we set 
\[
H_k:=\begin{cases} \frac{1}{6}\left((D^2-D+1)G_2-G_4\right)&\text{if }k=6,\\
\frac{1}{24}\left(-D^2G_{k-6}+\left(D^2+1\right)G_{k-4}-G_{k-2}\right)&\text{if }k\geq 8,
\end{cases}
\]
where the \begin{it}Eisenstein series\end{it} $G_k$ is defined by 
\[
G_k(\tau):=\frac{B_k}{2k}+\sum_{n\geqslant 1} \sigma_{k-1}(n)q^n
\]
for the Bernoulli numbers $B_k$ and $\sigma_r(n):=\sum_{d\mid n} d^{r}$.
\begin{cor}\label{cor:PrimeDetecting}
If $f\in \Omega$, then $f$ is a linear combination of the forms $D^nH_k$ for $n\geqslant 0$ and $k\geqslant 6$. Conversely, if a linear combination $f$ of the forms $D^nH_k$ with $n\geqslant 0$ and $k\geqslant 6$ has non-negative Fourier coefficients, then $f\in\Omega$.
\end{cor}

\begin{rem}
In the introduction of \cite{CvIO}, they conjecture that $\Q[n]$-linear combinations of the $M_a(n)$s detect primes if and only if they are $\Q[n]$-linear combinations of the forms in \cite[Table 1]{CvIO}. Given the classification of prime-detecting quasimodular forms in Corollary \ref{cor:PrimeDetecting}, their conjecture is equivalent to showing that a prime-detecting $\Q[n]$-linear combination of the $M_a(n)$ arises from a $\Q[n]$-linear combination of the forms $H_k$ with $k\geqslant 6$ if and only if it arises from a $\Q[n]$-linear combination of the forms in \cite[Table 1]{CvIO}.
\end{rem}
Theorem \ref{thm:PrimeDetecting} actually follows from a slightly stronger statement. Namely, let $\widetilde{\Omega}$ be the set of quasimodular forms (of mixed weights) which vanish at every prime, i.e., we have $f\in\widetilde{\Omega}$ if and only if $c_f(p)=0$ for every prime $p$. Then it turns out that our proof of Theorem \ref{thm:PrimeDetecting} implies that $\widetilde{\Omega}=\mathcal{E}\cap \widetilde{\Omega}$. Since $\Omega\subseteq \widetilde{\Omega}$, this implies Theorem \ref{thm:PrimeDetecting}. 

Since the space spanned by the $H_k$ is infinite-dimensional, it may not be obvious how to check whether a given form lies in this space. However, there exists a finite check whether elements of $\mathcal{E}$ lie in $\widetilde{\Omega}$ based on the growth of their Fourier coefficients.
\begin{thm}\label{thm:FiniteCheck}
Suppose that $f\in\mathcal{E}$. Let $r\in\N$ be such that for every $\varepsilon>0$ there exists $C_{\varepsilon}>0$ satisfying
\[
|c_f(n)|\leq C_{\varepsilon} n^{r-1+\varepsilon}
\]
for every $n\in\N$. Then $f\in\widetilde{\Omega}$ if and only if there exist $r$ primes $p_1,\dots,p_{r}$ for which
\[
c_f\left(p_j\right)=0.
\]
\end{thm}

\begin{rem}
From Theorem \ref{thm:FiniteCheck}, we see that the coefficients of an element of $\mathcal{E}$ either vanish at all primes or at finitely many of them. 
\end{rem}

The paper is organized as follows. In Section \ref{sec:prelims}, we introduce some preliminaries. In Section \ref{sec:signchanges}, we investigate the sign changes of Fourier coefficients of quasimodular cusp forms. We prove the main theorems, Theorem \ref{thm:PrimeDetecting} and Corollary \ref{cor:PrimeDetecting} in Section \ref{sec:main}. We finally prove Theorem \ref{thm:FiniteCheck} in Section \ref{sec:FiniteCheck}.

\section*{Acknowledgement}

The research was conducted during the conference HKU Number Theory Days 2025. The authors thank the Department of Mathematics at HKU and the Institute of Mathematical Research at HKU for supporting the conference and hosting the second author. The authors wish to thank Ken Ono, Jan-Willen van Ittersum, and Toshiki Matsusaka for their encouragement and helpful comments on an earlier version of this paper.

\section{Preliminaries}\label{sec:prelims}

\begin{it}Modular forms (for $\SL_2(\Z)$) of weight $k$\end{it} are holomorphic functions on the upper-half plane which satisfy 
\[
f\left(\frac{a\tau+b}{c\tau+d}\right)=(c\tau+d)^kf(\tau)
\]
for every $\left(\begin{smallmatrix}a&b\\c&d\end{smallmatrix}\right)\in\SL_2(\Z)$ and have moderate growth at cusps in the sense that they have Fourier expansions of the shape 
\[
f(\tau)=\sum_{n\geqslant 0} c_f(n)q^n.
\]
Those holomorphic modular forms which vanish at the cusps (i.e., $c_f(0)=0$) are called \begin{it}cusp forms\end{it}. We consider modular forms of full level (i.e. $\SL_2(\Z)$) throughout this paper.

The space of modular forms over $\SL_2(\Z)$ is a graded ring (graded by the weight) generated by $G_4$ and $G_6$ (as a ring). We call the space generated by $G_4$, $G_6$ and the weight two Eisenstein series $E_2$, the space of \begin{it}quasimodular forms\end{it}. This again has a grading (with the grading $2a+4b+6c$ for $E_2^aE_4^bE_6^c$), although we consider forms of mixed weight throughout this paper. The space of quasimodular forms is preserved under differentiation, and the space splits into the subspaces $\mathcal{E}$ of quasimodular Eisenstein series and the subspace $\mathcal{S}$ of \begin{it}quasimodular cusp forms\end{it} (i.e., the space spanned by cusp forms and their derivatives).

\section{Sign changes of quasimodular cusp forms}\label{sec:signchanges}
In this section, we show that Fourier coefficients of quasimodular cusp forms exhibit sign changes.
\begin{lem}\label{lem:signchanges}
Suppose that $F\in \mathcal{S}$ has a Fourier expansion 
\[
F(\tau)=\sum_{n\geq 1}c_F(n)q^n
\]
with $c_{F}(n)\in\R$. If $F\neq 0$, then the sequence $c_{F}(p)$, running over $p$ prime, has infinitely many sign changes. 
\end{lem}
\begin{proof}
Suppose $F\in \mathcal{S}$ is non-zero. Note that if $c_F(n)\geqslant 0$ for all $n\in\N$, then $|\sum_{p\leqslant x}c_F(p)| = \sum_{p\leqslant x}|c_F(p)| = \sum_{p\leqslant x} c_F(p)$. On the other hand, if the first sum grows slower as $x\to\infty$, then it must exhibit cancellation, and infinitely-many sign changes must occur. Based on this, we follow a standard argument where we compare the growth of
\[
\sum_{p\leqslant x} c_{F}(p)
\]
with 
\[
\sum_{p\leqslant x} |c_{F}(p)|^2
\]
as $x\to\infty$, showing that cancellation must occur in the first sum. Recalling that a quasimodular cusp form $F$ is a linear combination of cusp forms and their derivatives, and the space of modular forms has an orthonormal basis, we may express a quasimodular cusp form $F$ as 
\begin{eqnarray*}
    F= \sum_f \sum_j A_{f}(j) f^{(j)} 
\end{eqnarray*}
where $A_f(j)\in \mathbb{C}$ and $f^{(j)}=D^j f$ is the normalized $j$-th derivative of $f$. Moreover, we choose $f$'s to be (mutually orthogonal) Hecke eigenforms, which are primitive forms of not necessarily same weight for $SL_2(\mathbb{Z})$. We then have
\begin{eqnarray*}
    c_F(p) = \sum_f \sum_j A_{f}(j) p^j a_f(p) = \sum_{f} P_{f}(p) a_f(p)
\end{eqnarray*}
for some polynomials $P_{f}(x)\in \C[x]$. For every $f$, we let the weight of $f$ to be $k_{f}\geqslant 12$ and the degree of $P_{f}$ to be $j_{f} \geqslant 0$. The leading coefficient of $P_f$ is $A_f(j_f)$ which we simply denote by $A_{f}$. Recall, that from the work of Deligne, we know the Ramanujan bound for $f$, that is $|a_f(p)|\leqslant 2 p^{(k_f-1)/2}$. It then follows that for a given prime $p$, $P_{f}(p)a_f(p) = A_f p^{j_f} a_f(p) + O\left(p^{j_{f} + \frac{k_{f}-1}{2} -1}\right)$, where the implied constant depends at most on $F$.

Summing over the primes, and using the prime number theorem in this setting (see \cite[Theorem 5.13]{IwaniecKowalski}) we have
\[
\sum_{p\leqslant x} P_{f}(p)a_{f}(p) = A_{f}\sum_{p\leqslant x} p^{j_f} a_{f}(p) + O\left(\frac{x^{j_f + \frac{k_f-1}{2}}}{\log(x)}\right) = o\left(\frac{x^{j_{f} + \frac{k_{f}+1}{2}}}{\log(x)}\right).
\]
This gives us that 
\begin{align}\label{eqm}
\sum_{p\leqslant x} c_F(p) = o\left(\frac{x^{\alpha_0}}{\log(x)}\right)
\end{align}
where $\alpha_0:= \max_{f}\left\{\alpha_{f}\right\}$ with $\alpha_{f} := j_f + \frac{k_f +1}{2}$.

Next, 
\begin{eqnarray*}
|c_F(p)|^2 
&=& \sum_{f,g} P_f(p) \overline{P_g(p)} a_f(p)\overline{a_g(p)} \\
&=& \sum_f |P_f(p)|^2 |a_f(p)|^2 + \sum_{f\neq g} P_f(p) \overline{P_g(p)} a_f(p)\overline{a_g(p)}.     
\end{eqnarray*}
(In fact, $\overline{a_g(p)}=a_g(p)$ for all $p$.)

For every $f$ define $\beta_{f}$ to be $2\alpha_f - 1$.
From the Rankin-Selberg theory and from the prime number theorem in this context\footnote{
For example, with \cite[Theorem 5.19]{Gel} and \cite[Corollary 1.5]{LY}, 
$$ \sum_{p\leqslant x} \frac{a_f(p)\overline{a_g(p)}}{p^{(k_f+k_g)/2}} = \delta_{f=g}\cdot\big(\log\log x + c\big) 
+ O(\exp(-c'\sqrt{\log x}))
$$
where $\delta_{*}=1$ if $*$ is true or $0$ otherwise, and $c$ and $c'$ are some positive constants.}, 
we may deduce that 
\[
\sum_{p\leqslant x} P_{f}(p) \overline{P_{g}(p)} a_f(p) \overline{a_{g}(p)} = \delta_{f=g} \frac{|A_f|^2x^{\beta_{f}}}{\log(x)} + o\left(\frac{x^{\alpha_{f} + \alpha_{g}-1}}{\log(x)}\right)
\]
Let $\beta_0 := 2\alpha_0-1$, and let $M$ denote the set of forms $f$ such that $\beta_f = \beta_0$.

Observe that the non-vanishing of $F$ implies that there exists at least one $f\in M$ for which $A_f$ is non-zero. Plugging this all in we get,
\begin{align}\label{eqms}
\sum_{p\leqslant x} |c_F(p)|^2 = \frac{x^{\beta_{0}}}{\log(x)}\left(\sum_{f\in M} |A_f|^2 + o(1)\right) \gg \frac{x^{\beta_{0}}}{\log(x)}.
\end{align}

It follows from Deligne's bound that for all prime $p$, 
\[
|c_F(p)|\leqslant \sum_{f} |P_{f}(p) a_f(p)| \leqslant \sum_{f} \|P_f\| p^{j_f}|a_f(p)| \leqslant C_F p^{\alpha_0-1}
\]
where $\|P\|=\sum_{r=0}^{m} |A_r|$ if $P(x)=\sum_{r=0}^m A_rx^r\in \mathbb{C}[x]$ and $C_F>0$ is a constant.  This yields
\[
\sum_{p\leqslant y} |c_F(p)|^2 = O\left(\frac{y^{2\alpha_0-1}}{\log(y)}\right).
\]

We prove the infinitude of sign changes in $\{c_F(p)\}$ by contradiction. Suppose without loss of generality that $c_F(p) \geqslant 0$ for all $p \geqslant y$ for some fixed $y$. 

With \eqref{eqm}, we have
\[
\sum_{p\leqslant x} |c_F(p)|^2 \leqslant  C_F x^{\alpha_0-1} \sum_{y\leqslant p\leqslant x} c_F(p) + O\left(\frac{y^{2\alpha_0-1}}{\log(y)}\right)= o\left(\frac{x^{2\alpha_0-1}}{\log(x)}\right)
\]
as $x\to \infty$. This contradicts to \eqref{eqms}, for $\beta_0=2\alpha_0-1$. Our proof is complete. 


\end{proof}

\begin{rem}
    Standard modifications of the above argument maybe used to give a lower bound on the number of sign changes. For our purposes, the existence of infinitely many sign changes suffices and hence we do not pursue this here.
\end{rem}

\section{Proof of the main theorem}\label{sec:main}
Before proving Theorem \ref{thm:PrimeDetecting}, we require a lemma about Eisenstein series, which essentially shows that there are at most finitely many sign changes. We define $\sgn(0)=0$.

\begin{lem}\label{lem:Eisenstein}
Suppose that $f\in\mathcal{E}$ has the Fourier expansion 
\[
f(\tau)=\sum_{n\geqslant 0} c_f(n)q^n
\]
with $c_f(n)\in\R$. Then there exists $\varepsilon\in\{-1,0,1\}$ such that for all sufficiently large prime $p$ we have 
\[
\sgn(c_f(p))=\varepsilon.
\]
\end{lem}

\begin{proof}
We note first that 
\[
D^{\ell}G_k(\tau)=\delta_{\ell=0}\frac{B_k}{2k} + \sum_{n\geqslant 1} n^{\ell}\sigma_{k-1}(n)q^n.
\]
Since
\[
f=\sum_{k,\ell}\alpha_{k,\ell} D^{\ell}G_k
\]
for some $\alpha_{k,\ell}\in\C$. For $n\in\N$ the $n$-th Fourier coefficient of $f$ is 
\[
c_f(n)=\sum_{k,\ell}\alpha_{k,\ell} n^{\ell}\sigma_{k-1}(n).
\]
For $n=p$ prime, we thus have 
\[
c_f(p)=\sum_{k,\ell}\alpha_{k,\ell} p^{\ell}\left(1+p^{k-1}\right).
\]
We next combine all like powers of $p$ to obtain 
\[
c_f(p)=\sum_{r} \beta_r p^r
\]
for some $\beta_r$ (since $c_f(p)\in\R$, we indeed have $\beta_r\in\R$). Note that $\beta_r$ is some linear combination of the values $\alpha_{k,\ell}$, and hence only depends on $f$. If $\beta_r=0$ for every $r$, then $c_f(p)=0$ for all $p$. So the claim holds for $\varepsilon=0$ in this case.

On the other hand, if $\beta_r\neq 0$ for some $r$, then we may choose $r_0$ maximal so that $\beta_{r_0}\neq 0$. Then as $p\to\infty$ we have 
\[
c_f(p)=\beta_{r_0}p^{r_0}+O\left(p^{r_0-1}\right),
\]
where the implied constant only depends on $f$. We see that for $p$ sufficiently large we have 
\[
\sgn\left(c_f(p)\right)=\sgn\left(\beta_{r_0}\right)=:\varepsilon.
\]
This yields the claim.
\end{proof}
We are now ready to prove Theorem \ref{thm:PrimeDetecting}.
\begin{proof}[Proof of Theorem \ref{thm:PrimeDetecting}]
Suppose that $F\in \widetilde{\Omega}$.  Let us suppose that $F = \sum_{\ell, k,j} \alpha_{\ell,k,j} D^{\ell} f_{k,j}$, where $f_{k,j}$ is either $G_k$ or a cuspidal Hecke eigenform of weight $k$. Normalizing the Hecke eigenforms to have real coefficients, we see that we may write $F=F_{\re}+iF_{\Im}$ with
$F_{\re} := \sum_{\ell,k,j} \re(\alpha_{\ell,k,j}) D^{\ell} f_{k,j}$ and $F_{\Im} := \sum_{\ell,k,j} \Im(\alpha_{\ell,k,j}) D^{\ell} f_{k,j}$. In particular, both $F_{\re}$ and $F_{\Im}$ have real Fourier coefficients. Also observe that $F\in \widetilde{\Omega}$ if and only if both $F_{\re}$ and $F_{\Im}$ belong to $\widetilde{\Omega}$. 

To prove the theorem, it suffices to show that both $F_{\re}$ and $F_{\Im}$ belong to $\mathcal{E}$. We shall show this for $F_{\re}$, with the other case being identical. To ease notation, we shall denote $F_{\re}$ simply as $F$.
We split 
\[
F=F_E+F_S
\]
with $F_E\in\mathcal{E}$ and $F_S\in\mathcal{S}$. Since $F_E\in\mathcal{E}$ with real Fourier coefficients, Lemma \ref{lem:Eisenstein} implies that there exists $\varepsilon\in\{-1,0,1\}$ such that for $p$ sufficiently large we have 
\[
\sgn\left(c_{F_E}(p)\right)=\varepsilon.
\]
Since $c_F(p)=0$ for all $p$, for $p$ sufficiently large we have 
\[
\sgn\left(c_{F_S}(p)\right)=-\sgn\left(c_{F_E}(p)\right)=-\varepsilon.
\]
Thus $c_{F_S}(p)$ only has at most finitely many sign changes. By Lemma \ref{lem:signchanges}, this implies that $F_S=0$. Therefore $F=F_E\in\mathcal{E}$. So $F\in \mathcal{E}\cap \widetilde{\Omega}$, and we see that $\widetilde{\Omega}=\mathcal{E}\cap\widetilde{\Omega}$.\qedhere
\end{proof}
Corollary \ref{cor:PrimeDetecting} now follows immediately.
\begin{proof}[Proof of Corollary \ref{cor:PrimeDetecting}]
By Theorem \ref{thm:PrimeDetecting}, if $f\in\Omega$, then $f\in \mathcal{E}\cap\Omega$. The claimed classification of $\mathcal{E}\cap\Omega$ is given in \cite[Theorem 2.3]{CvIO}.
\end{proof}
\begin{rem}The proof of Theorem \ref{thm:PrimeDetecting} and Corollary \ref{cor:PrimeDetecting} relies on sign changes for Fourier coefficients of quasimodular cusp forms in the subsequence of primes and the absence of sign changes for Fourier coefficients of quasimodular Eisenstein series in the same subsequence. In principle, if there is some other infinite subset $A\subseteq\N$ for which the quasimodular Eisenstein series lack sign changes and the quasimodular cusp forms exhibit them, then one should be able to argue that only the Eisenstein series can detect the set $A$. Following this idea, we plan to investigate the detection of primes in arithmetic progressions in a forthcoming paper.
\end{rem}
\section{Finite Checks for prime-vanishing forms}\label{sec:FiniteCheck}
In this section, we show how to check vanishing at a finite number of primes to determine if an element lies in $\widetilde{\Omega}$ or not.

\begin{proof}[Proof of Theorem \ref{thm:FiniteCheck}]
Since $f\in \mathcal{E}$, we may write it as
\[
f = \sum_{\ell,k} \alpha_{\ell,k} D^{\ell} G_k = \sum_{n\geqslant 0} c_f(n) q^n.
\]
Since the coefficients of $D^{\ell}G_k$ grow at a rate (where $a(n)\ll_{\varepsilon} b(n,\varepsilon)$ means that for every $\varepsilon>0$ there exists a constant $C_{\varepsilon}>0$ depending only on $\varepsilon$ such that $|a(n)|\leq C_{\varepsilon} |b(n,\varepsilon)|$)
\[
n^{\ell}\sigma_{k-1}(n)\ll_{\varepsilon} n^{\ell+k-1+\varepsilon},
\]
we conclude that $\ell+k\leq r$. Moreover, we may choose $r$ to be the maximum of $\ell + k$ for which $\alpha_{\ell,k}\neq 0$. In particular, $r$ is finite.

As in the proof of Lemma \ref{lem:Eisenstein}, we may rearrange to obtain
\[
c_f(p) = \sum_{j=0}^{r} \beta_j p^j.
\]
Now, if $c_f(p_j) = 0$ for $1\leqslant j\leqslant r$, then we obtain the system of equations
\[
\begin{pmatrix}
    1&p_1&p_1^2&\ldots&p_1^r\\
    1&p_2&p_2^2&\ldots&p_2^r\\
    \vdots&\vdots&\vdots&\ddots&\vdots\\
    1&p_{r+1}&p_{r+1}^2&\ldots&p_{r+1}^r
\end{pmatrix}
\begin{pmatrix}
    \beta_0\\ \beta_1\\ \vdots \\ \beta_r
\end{pmatrix} = 0.
\]
This system, being a Vandermonde system, is (uniquely) solvable, and hence $\beta_0 = \beta_1 = \ldots = \beta_r = 0$. Hence the claim follows.
\end{proof}

\end{document}